\documentclass[12pt]{amsart}
\usepackage{amsmath}
\usepackage{graphicx,xcolor}
\usepackage{kantlipsum} 
\usepackage{float}
\usepackage{xcolor}
\usepackage{underoverlap}
\usepackage{ytableau}
\usepackage{oubraces}
\usepackage[makeroom]{cancel}
\usepackage{tikz,tikz-cd}
\usepackage[utf8]{inputenc}
\usepackage{a4wide}
\usepackage{amssymb}
\calclayout

\usepackage{stmaryrd}
\usepackage{tikz}
\newtheorem{theorem}{Theorem}[section]

\newtheorem{corollary}[theorem]{Corollary}
\theoremstyle{definition}
\newtheorem{definition}[theorem]{Definition}

\DeclareMathOperator{\corank}{corank}

\DeclareMathOperator{\Hom}{Hom}

\def\max{\mathrm{max}} 
\newcommand{\kk}{\mathsf k}
\newcommand{\ello}{l}
\newcommand{\kt}{\kk\llbracket t\rrbracket}

\title{Identifying partitions with maximum commuting orbit $Q=(u,u-r)$}
\author{Mats Boij}
\address{Department of Mathematics, KTH Royal Institute of Technology, SE-100 44 Stockholm, Sweden}\email{boij@kth.se}
\author{Anthony Iarrobino}
\address{Department of Mathematics, Northeastern University, Boston, MA 02115, USA}
\email{a.iarrobino@neu.edu}
\author{Leila Khatami}
\address{Department of Mathematics, Union College, Schenectady, NY 12309, USA}
\email{khatamil@union.edu}
\date{\today}
\begin{document}
\begin{abstract}
The authors here show that the partition $P_{k,\ello}(Q)$ in the table $\mathcal 
T(Q)$ of partitions having maximal nilpotent commutator a given stable partition $Q$, defined in \cite{IKVZ2}, is identical to the analogous partition $P_{k,\ello}^Q$ defined by the authors in \cite{BIK} using the Burge correspondence.

\end{abstract}\maketitle
\section{Introduction}
The Jordan type $P_B$ of a nilpotent $n\times n$ nilpotent matrix $B$ is the partition giving the Jordan block decomposition of a Jordan block matrix $J_P$ similar to $B$. Since the variety $\mathcal N_B$ of nilpotent matrices commuting with $B$ is irreducible, there is a generic commuting Jordan type $\mathcal D(P)$ of a matrix $A$ in $\mathcal N_B$, which is maximum in the 
dominance order.  The partition $\mathcal D(P)$ is stable, that is, $\mathcal D(\mathcal D(P))=\mathcal D(P)$; equivalently, $\mathcal D(P)$ has parts that differ pairwise by at least two - known in combinatorics as super-distinct. In addition the largest and smallest parts of $\mathcal D(P)$ were determined by P. Oblak and by the third author in\cite{Ob,Kh1}: so in 2014 $
\mathcal D(P)$ was known when $\mathcal D(P)$ has three or less parts.
 R. Basili in 2022 determined $\mathcal D(P)$ in general, proving the 2012 recursive conjecture of P.~Oblak, that had been widely studied with many partial results. The second two authors with B. Van Steirteghem and R.~ Zhao showed that
the partitions in $\mathcal D^{-1}(Q)$ for a stable $Q=(u,u-r)$ form a $(r-1)\times (u-r)$ box, such that $P_{k,\ello}(Q)$ has $(k+\ello)$ parts. Their arrangement used what they termed $A$ rows, and $B,C$ hooks \cite{IKVZ2}.
The box theorem of 
J. Irving, T. Ko\v{s}ir and M. Mastnak \cite{IKM} shows that for each stable partition $Q$ the Jordan types $P$ satisfying $\mathcal D(P)=Q$ form a box, whose entries could be determined by their Burge codes. For an account of these developments, see the third author's survey \cite{Kh2}.  \par
In \cite{BIK} the authors used the Irving, Ko\v{s}ir, Mastnak box theorem in the special case that $Q$ has two parts, to determine the equations $E^Q_{k,\ello}$ defining the locus $P_B=P^Q_{k,\ello}$ in $\mathcal N(Q)$. There, $P^Q_{k,\ello}$ in $\mathcal D^{-1}(Q)$ is determined by its Burge code.
We here prove that the partition $P_{k,\ello}(Q)$ in the table $\mathcal T(Q)$ from  \cite{IKVZ2} is identical to 
the partition $P^Q_{k,\ello}$ in ${\mathcal D}^{-1}(Q)$ as defined in \cite{BIK}: this is \cite[Proposition 3.9]{BIK}, whose proof we deferred to here. We show this in Theorem~\ref{equTOpar} and Corollary \ref{identifycor} below.   We prove this by showing that the equations $E^Q_{k,\ello}$ defining each locus are the same. In order for this note to be self-contained, there is overlap between Section \ref{sec3} here and the Notations and Background Section 2 of \cite{BIK}.
\section{The Table $
\mathcal T(Q)$ from \cite{IKVZ2}}\label{sec2}
We need the definition of types A,B, and C partitions, elements of the table $\mathcal T(Q), Q=(u,u-r)$ giving $\mathcal D^{-1}(Q),$ from \cite{IKVZ2}.
Recall that a partition is \emph{almost rectangular} if its largest and smallest part differ by at most one.
The stable partition $\mathcal D(P)$ has two parts exactly when $P$ is the union of two almost rectangular partitions, but $P$ is not itself almost rectangular. Hence there exist $a,b\in \mathbb N$ with $a-b\ge 2$ such that
\begin{equation}\label{Pabeq}P=\left(a^{n_a},(a-1)^{n_{a-1}}, b^{n_b}, (b-1)^{n_{b-1}}\right) \text {with $n_a>0$ and $n_b>0$}.
\end{equation} 
We will assume that $n_{b-1}=0$ if $b=1$. We denote here by $n_i$ the number of parts of $P$ having length $i$.
\begin{definition}[Type A,B,C partitions in $\mathcal D^{-1}(Q)$]\label{casesdef}
Let $Q=(u,u-r)$ with $u,r\in \mathbb N, u>r\ge 2$ and let $P\in \mathcal D^{-1}(Q)$ satisfy \eqref{Pabeq}.\par
We say that $P$ is of \emph{type A} if $u=a\cdot n_a+(a-1)n_{a-1}$;\par
We say that $P$ is of \emph{ type B} if $u=2n_{a}+2n_{a-1}+bn_b+(b-1)n_{b-1}$, or if $b=a-2, n_{b-1}=0 $ and $u=2n_a+(a-1)n_{a-1}+bn_b$. \par
We say that $P$ is of \emph{type C} if $b=a-2$, if each of $n_a,n_{a-1},n_b,n_{b-1}$ is non-zero,  and $u=2n_{a}+(a-1)n_{a-1}+bn_b$. 
\end{definition}
We now state part I of the Table Theorem from \cite{IKVZ2}. In it, for compactness, we will use $P_{k,\ello}$ to denote the partition $P_{k,\ello}(Q)$ in the table $\mathcal T(Q)$.
\begin{theorem}[Table Theorem, part I]\label{tablethm} Let $Q=(u, u-r)$ with $u>r\geq 2$. 
\begin{itemize}

\item[(a)] For every positive integer $t$ such that $1\le t\le \min\{u-r,\lfloor \frac{r-1}{2}\rfloor\}=t_{\max}$,  define the set 
\begin{equation*}
A_t=\{(k,\ello )\in \mathbb{N}\times\mathbb{N}\,|\, k_{t-1} \le k<k_{t}\mbox { and }t\le \ello\le u-r\}.
\end{equation*}

For $T=\min\{u-r,\lfloor \frac{r-1}{2}\rfloor\}+1=t_{\max}+1$ define the set 
\begin{equation*}
A_T=\{(k,\ello)\in {\mathbb{N}\times\mathbb{N}}\,
{\mbox{ such that }}\, k=k_{t_\max}=r-T{\text { and }}T\le \ello \leq u-r\}.
\end{equation*}
Moreover 
\begin{equation}\label{ATnonempeq}
A_T\not=\emptyset \Leftrightarrow r \text{ is even and }u\ge \frac{3r}{2},
\end{equation}
and in that case $T={k_{t_\max}}={\frac{r}{2}}$.

 Then for all $(k, \ello)\in A_t$, when $t\in \{1,2,\ldots ,T\}$ the partition $P_{k,\ello}=\left([u]^{k},[u-r]^{\ello}\right)$ is of type A and
satisfies $\mathcal D(P_{k, \ello})=(u,u-r)$.

\item[(b)] For every positive integer $t$ such that $1\le t\le\min\{u-r,\lfloor \frac{r-1}{2}\rfloor\}$, define the subset $C_t\subset  \mathbb{N}\times\mathbb{N}$ as 
$$C_t=\{(k,t)\,|\, k_t\leq k < c_t\}.$$ 
 Then for all $(k, \ello)\in C_t$, the partition $$P_{k,\ello}=\left([u-r+2t]^{t},[u-2t-d_t(q_t-1)]^{k-d_t}, (q_t-1)^{d_t})\right)$$
 is of type C but not of type A or B, and  satisfies $\mathcal D(P_{k, \ello})=(u,u-r)$.

\item[(c)] For every positive integer $t$ such that $1\le t\le\min\{u-r,\lfloor \frac{r-1}{2}\rfloor\}$, define the subset $B_t\subset \mathbb{N}\times\mathbb{N}$ as $$B_t=\{(k,t)\,|\, c_t\leq k \leq r-t\} \cup \{(r-t,\ello)\,|\,t< \ello \leq u-r\}.$$ 
Then for all $(k, \ello)\in B_t$, the partition $$P_{k,\ello}=\left([u-r+2t]^{t},[u-2t]^{k+\ello-t}\right)$$
 is of type B but not of type A and satisfies $\mathcal D(P_{k, \ello})=(u,u-r)$.

\item[(d)] Each pair $(k,\ello)\in \mathbb N\times \mathbb N$ with $1\le k\le r-1$ and $ 1\le \ello\le u-r$ belongs to one and only one set $A_t$, $B_t$ or $C_t$ defined above. In particular there are listed above $(r-1)(u-r)$ distinct partitions  $\{P_{k,\ello}\}$, each satisfying $\mathcal D(P_{k, \ello})=(u,u-r)$. Moreover, every partition $P_{k,\ello}$ has $k+\ello$ parts.

\end{itemize}

\end{theorem}
Part II of the Table Theorem is the completeness of the table \cite[Theorem 3.19]{IKVZ2}.
\begin{theorem}
Let $Q= (u,u-r), u>r \ge 2$. The table $\mathcal T(Q)$ comprising the partitions $P_{k,\ello}(Q)$ of Theorem \ref{tablethm} contains all the partitions in $\mathcal D^{-1}(Q)$.
\end{theorem}
\section{The table of partitions $P^Q_{k,\ello}$ of \cite{BIK}}\label{sec3}
We defined the partitions $P^Q_{k,\ello}$ of the table $\mathcal D^{-1}(Q)$ in \cite{BIK} using the Burge code, following \cite{IKM}. 

When $Q = (u,u-r)$ where $u>r\ge 2$ we will denote the entries in this $(r-1)\times (u-r)$ table of partitions $\mathcal D^{-1}(Q)$ by $P^Q_{k,\ello}$ where $1\le k\le r-1$ and $1\le \ello\le u-r$. The Burge code of $P\in \mathcal D^{-1}(Q)$ is a word in the letters $\alpha,\beta$ ending in $\alpha$ and having here two adjacent pairs $(\beta\cdot\alpha)$ see \cite[Section 2]{BIK},\cite[Section 2.2]{IKM}, or \cite{Bur}. The Burge code for $P^Q_{k,\ello}$ is
\begin{equation}\label{Burgespecialeq}
    \alpha^{u-r-\ello}\beta^{\ello}\alpha^{r-k}\beta^k\alpha.
\end{equation}
We will show the equality of the partition $P_{k,\ello}(Q)$ in $\mathcal T(Q)$ (from \cite{IKVZ2}) and $P^Q_{k,\ello}$ in $\mathcal D^{-1}(Q)$ (Burge code from \cite{BIK}) by showing that the equations defining their loci in $\mathcal N(Q)$ are the same (Corollary \ref{identifycor}).

\subsection{The matrices in the nilpotent commutator of $Q=(u,u-r)$.}
When $Q = (u,u-r)$ with $u>r\ge 2$, we have that the matrices in $\mathcal N_{J_Q}$ can be written as 
\begin{equation}\label{B-matrix}
B = \begin{bmatrix}
    0&a_1&a_2&a_3&\cdots&a_r&a_{r+1}&\cdots&a_{u-1}&g_0&g_1&g_2&\dots&g_{u-r-1}\\
    0&0&a_1&a_2&\cdots&a_{r-1}&a_{r}&\cdots &a_{u-2}&0&g_0&g_1&\dots&g_{u-r-2}\\
    \vdots&\vdots&\vdots&\vdots&\ddots&\vdots&\vdots&\ddots&\vdots&\vdots&\vdots&\vdots&\ddots&\vdots \\
    0&0&0&0&\cdots&a_{2r-u+1}&a_{2r-u+2}&\cdot&a_{r}&0&0&0&\cdots&g_0\\
    \vdots&\vdots&\vdots&\vdots&\ddots&\vdots&\vdots&\vdots&\vdots&\ddots&\vdots \\
    0&0&0&0&\cdots&0&0&\cdots&a_1&0&0&0&\dots&0\\
    0&0&0&0&\cdots&0&0&\cdots&0&0&0&0&\dots&0\\
    0&0&0&0&\cdots&h_0&h_1&\cdots&h_{u-r-1}&0&b_1&b_2&\dots&b_{u-r-1}\\
    0&0&0&0&\cdots&0&h_0&\cdots&h_{u-r-2}&0&0&b_1&\dots&b_{u-r-2}\\
    \vdots&\vdots&\vdots&\vdots&\ddots&\vdots&\vdots&\ddots&\vdots&\vdots&\vdots&\vdots&\ddots&\vdots \\
    0&0&0&0&\cdots&0&0&\cdot&h_{0}&0&0&0&\cdots&0\\
\end{bmatrix}
\end{equation}
Using this we can get an affine coordinate ring of $\mathcal N_{J_Q}$ as \[\kk[\mathcal N_{J_Q}] = \kk[a_1,a_2,\dots,a_{u-1},b_1,b_2,\dots,b_{u-r-1},g_0,g_1,\dots,g_{u-r-1},h_0,h_1,\dots,h_{u-r-1}].\]
\par
We now recall the definition of the auxiliary matrix $M_B$ from \cite{BIK}.
For a matrix $B$ in $\mathcal N_{J_Q}$ we can associate to $B$, as in (\ref{B-matrix}), a matrix 
\begin{equation}\label{MBeqn}
M_B = \begin{bmatrix}
    a&g\\
    ht^r&b
\end{bmatrix} = \begin{bmatrix}
    a_1t+a_2t^2+\cdots+a_{u-1}t^{u-1}&g_0+g_1t+\cdots+g_{u-r-1}t^{u-r-1}\\
    h_0t^r+a_1t^{r+1}+\cdots+h_{u-r-1}t^{u-1}&g_1t+b_2t^2+\cdots+b_{u-r-1}t^{u-r-1}\\
\end{bmatrix}
\end{equation}
where $a,ht^r\in \kk[t]/(t^u)$ and $g,b\in \kk[t]/(t^{u-r})$. This matrix represents a homomorphism in \[
\Hom_{\kk[t]}(\kk[t]/(t^u)\oplus \kk[t]/(t^{u-r}),\kk[t]/(t^u)\oplus \kk[t]/(t^{u-r}))\] and the map 
\[
    \mathcal N_{J_Q} \longrightarrow \mathcal R_Q:=\Hom_{\kk[t]}(\kk[t]/(t^u)\oplus \kk[t]/(t^{u-r}),\kk[t]/(t^u)\oplus \kk[t]/(t^{u-r}))
\]
given by $B\mapsto M_B$ is a ring isomorphism.
We will use the power series ring $\kt$ for calculations and then restrict to the quotients by $(t^u)$ and $(t^{u-r})$.\par

In the endomorphism ring $\Hom_{\kt}(\kt^2,\kt^2)$, there is a subring given by 
\begin{equation}\label{eq:Hr}
    \mathcal H_r = \left\{\begin{bmatrix}
    a&g\\
    ht^r&b
\end{bmatrix}\colon a,b,g,h\in \kt\right\}
\end{equation}
and the natural map $\mathcal H_r \longrightarrow\mathcal R_Q$ is a surjective ring homomorphism. 
\begin{definition}\label{ordermatrixdef} The \emph{order matrix} of $B$
is  
\[T=  \begin{bmatrix}
	k&&0\\r&&\ello'
	 \end{bmatrix}.\]
where $k$ is the order of $a$, $\ello$ is the order of $b$ in $M_B$ and $\ello'=\ello\oplus (r-k)$.
\end{definition}

\subsection{The equations $E_{k,\ello}$}

\begin{definition}\label{eqndef}
    For $Q = (u,u-r)$, where $u>r\ge 2$, and $1\le k\le r-1$ and $1\le \ello \le u-r$ let 
    \[
    E^Q_{k,\ello} = \{a_1,a_2,\dots,a_{k-1},b_1,b_2,\dots b_{\ello-1}\} \subseteq \kk[\mathcal N_{J_Q}]
    \]
    when $k+\ello\le r$; and let 
   \[\begin{split}
        E^Q_{k,\ello} &= \{a_1,a_2,\dots,a_{k-1},b_1,b_2,\dots b_{r-k-1},a_kb_{r-k}-g_0h_0, \\&a_kb_{r+1-k}+a_{k+1}b_{r-k}-g_0h_1-g_1h_0,\dots,\sum_{j=0}^{k+\ello-r-1}(a_{k+j}b_{r-k-j}-g_jh_{k+\ello-r-1-j})\}\subseteq \kk[\mathcal N_{J_Q}]
   \end{split}
    \]
    if $k+\ello\ge r+1$.    Let $I^Q_{k,\ello} = \left(E^Q_{k,\ello}\right)$ be the ideal generated by $E^Q_{k,\ello}$ in $\kk[\mathcal N_{J_Q}]$ and let $X^Q_{k,\ello} = V(I^Q_{k,\ello})$ be the variety defined by it.
\end{definition}

\section{Equation to partition}
We show that the equations $E^Q_{k,\ello}$ holding for the entries of a matrix $A\in\mathcal N(Q)$, and that $A$ is general enough, determine that $P_A=P_{k,\ello}(Q)$ in the table $\mathcal T(Q)$. This in effect clarifies and shows one direction of \cite[Conjecture~4.19]{IKVZ1}, made in 2015, that the locus of partitions $P_{k,\ello}(Q)$ in $\mathcal N(Q)$ is given by the equations $E^Q_{k,\ello}.$  
\begin{theorem}\label{equTOpar}
Let $A$ be a matrix in the nilpotent commutator of the Jordan matrix $J_Q, Q=(u,u-r), r\ge 2$. Let $(k,\ello)$ satisfy $1\leq k\leq r-1$ and $1\leq \ello\leq u-r$. Assume that the equations $E^Q_{k,\ello}$ hold for entries of $A$, and that the rest of the entries of $A$ are generic. Then the Jordan type of $A$ is equal to the partition $P_{k,\ello}(Q)$ in the table $\mathcal{T}(Q), Q=(u,u-r)$ of Theorem \ref{tablethm}.

\end{theorem}
\begin{proof}
The order matrix of $A$ is 
\[T=  \begin{bmatrix}
	k&&0\\r&&\ello'
	 \end{bmatrix}.\]
where $\ello'=\ello\oplus (r-k)$. The equations $E^Q_{k,\ello}$ also imply that $\corank A=k+\ello$.

By \cite[Lemma 3.6]{BIK} for $s\geq 1$, 
\[
\corank A^s=(k+\ello)s\oplus ((T^{\otimes s})_{11}+u-r)\oplus(u+u-r).
\]

Moreover, by \cite[Lemma 3.5]{BIK} for $s\geq 2$, 
\[
(T^{\otimes s})_{11}=\left\{
	\begin{array}{ll}
	sk\oplus\frac{s}{2}r\oplus((s-2)\ello'+r)\oplus u&\mbox{ if } s \mbox{ is even},\\ \\
	sk\oplus(k+\frac{s-1}{2}r)\oplus(\ello'+\frac{s-1}{2}r)\oplus((s-2)\ello'+r)\oplus u&\mbox{ if } s \mbox{ is odd}.
	\end{array}
		\right.
\]

Since $k+\ello'\leq r$, if $k\leq \ello'$ then $2k\leq r$, and therefore
\[
\begin{array}{ll}
2ks&\leq rs\\
\mbox{ and}\\
2ks&=2k+2(s-1)k\\
&\leq 2k+(s-1)r\\ 
&\leq 2\ello'+(s-1)r.
\end{array}
\]


On the other hand, if $\ello'<k$, then $2\ello'<r$ and we have

\[
\begin{array}{ll}
2((s-2)\ello'+r)&< sr\\
\mbox{ and}\\
2((s-2)\ello'+r)&=2\ello'+2(s-3)\ello'+2r\\
&< 2\ello'+(s-1)r \\
&<2k+(s-1)r.
\end{array}
\]
Thus for $s\geq 2$ we always have $(T^{\otimes s})_{11}=sk\oplus((s-2)\ello'+r)\oplus u.$ 

Consequently, 
\[\begin{array}{ll}
\corank A&=k+\ello,\\ 
\corank A^s&=s(k+\ello)\oplus (sk+u-r)\oplus((s-2)\ell'+u)\oplus(u+u-r), \mbox{ for }s\geq 2.
\end{array}\]
\bigskip

In the $sy$-plane, we consider the following four lines corresponding to the components of $\corank A^s$.
\[\begin{array}{ll}
  L_1(s)&=(k+\ello)s\\
  L_2(s)&=ks+u-r\\
  L_3(s)&=\ell's+u-2\ell'\\
  L_4(s)&=u+u-r\\
\end{array}
\]

For $s=1$ the corank of $A^s$ is obtained by $L_1$ and for $s$ large enough, it stabilizes at $L_4$. For $s\geq 2$ the corank of $A^s$ is determined by the line that lies below the other three at that particular value of $s$. For integers, $i$ and $j$ with $0\leq i<j\leq 3$, we denote the $s$-coordinate of the intersection of $L_i$ and $L_j$ by $x_{ij}$. In particular, we have 

\[\begin{array}{llll}
x_{1,2}&=\dfrac{u-r}{\ell},&
x_{1,3}&=\dfrac{u-2\ell'}{k+\ell-\ell'},\\\\
x_{2,3}&=\dfrac{r-2\ell'}{k-\ell'}, &
x_{2,4}&=\dfrac{u}{k},\\\\
x_{3,4}&=\dfrac{u-r}{\ell'}+2.
\end{array}\]

We first study the Jordan partition of a matrix $A$ based on the ``path" through which the corank of powers of $A$ are obtained. 
\bigskip

\noindent\underline{\bf Case A. ($L_1 \to L_2 \to L_4$)} Suppose that $L_3(s)\geq L_!(s)\oplus L_2(s)\oplus L_4(s)$, for $s\geq 2$. Then 
\begin{equation}\label{smallkequ}
\begin{array}{ll}
\corank A^s&=s(k+\ello)\oplus (sk+u-r)\oplus(u+u-r)\\
           &=L_1(s)\oplus L_2(s)\oplus L_4(s).
\end{array}
\end{equation}
Let 
\[\begin{array}{lcl}
\begin{array}{ll}
s_1&=\lfloor x_{1,2}\rfloor=\lfloor\frac{u-r}{\ello}\rfloor,\\
e_1&=u-r-s_1\ello.
\end{array}
&\mbox{and}&
\begin{array}{lll}
s_2&=\lfloor x_{2,4}\rfloor=\lfloor\frac{u}{k}\rfloor,\\
e_2&=u-s_2k,
\end{array}
\end{array}\]

Then by equation~\ref{smallkequ} we have
\[
\corank A^s=\left\{
	\begin{array}{ll}
	s(k+\ello),& \mbox{ if }s=1, \dots, s_1\\
	sk+u-r,&\mbox{ if }s= s_1+1, \dots, s_2\\
	u+u-r,&\mbox{ if }s\geq  s_2+1.
	\end{array}
	\right.\]

We note that by assumption, for all $s$, we have $\corank A^s\leq L_3(s)$. In particular, $u+u-r\leq L_3(x_{2,4})$. Equivalently the inequality $\frac{u-r}{\ello'}+2\leq \frac{u}{k}$ holds. Since by definition $\ello'\leq \ello$, we also have
$\frac{u-r}{\ello}+2\leq \frac{u}{k}$. Thus $s_1+2 \leq s_2$. See figure~\ref{caseAfig} for a visualization of the linear terms affecting $\corank A^s$ in this case.

\begin{figure}

\boxed{\includegraphics[scale=.3]{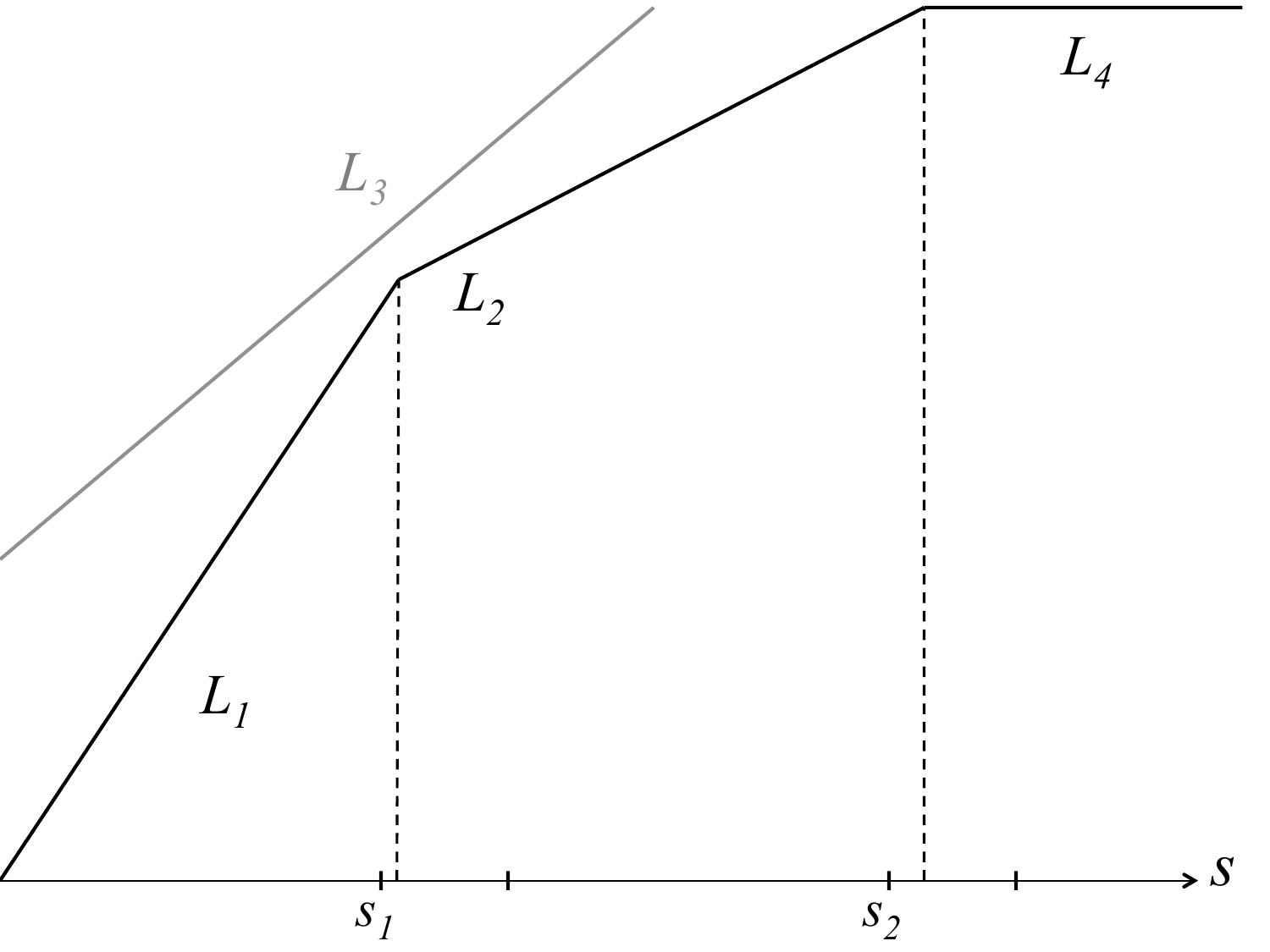}}

\caption{Comparing terms in the expression for $\corank A^s$ in Case A of the proof of Theorem~\ref{equTOpar}.}\label{caseAfig}
\end{figure}

Let $P$ be the Jordan type of matrix $A$. Then 
$$P^{\vee}=\Big(\, (k+\ello)^{s_1}, \, k+e_1 , \,
k^{s_2-s_1-1},  \,  e_2\Big),$$
and therefore
\[
  P=\left\{\begin{array}{ll}
             \Big(\, (s_2+1)^{e_2},\, s_2^{k-e_2}, \, (s_1+1)^{e_1}, \, s_1^{\ello-e_1}\, \Big),& \mbox{if }e_1,e_2>0\\\\

             \Big(\, (s_2+1)^{e_2},\, s_2^{k-e_2}, \, s_1^{\ello}\, \Big),& \mbox{if }e_2>0 \mbox{ and }e_1=0\\ \\

            \Big(\, s_2^{k}, \, (s_1+1)^{e_1}, \, s_1^{\ello-e_1}\, \Big),& \mbox{if }e_2=0 \mbox{ and }e_1>0\\ \\

             \Big(\, s_2^{k},  \, s_1^{\ello}\, \Big),& \mbox{if } e_1=e_2=0.

           \end{array}\right.
       \]

In all cases above, we can write $P=([u]^k,[u-r]^\ello)$. Looking at $U$-chains in $P$, we have
\[\begin{array}{ll}
    |U_{top}| &=u  \\
    |U_{bottom}|&=u-r+2k. 
\end{array}
\]


We claim that $u-r+2k\leq u$. In fact, if $u-r+2k> u$ then $2k>r$. By definition we also know that $r\geq k+\ello'$. Thus the inequality $2k>r$ implies that $k>\ello'$. However, if this is the case, then for all $s\geq 2$, we have
\[\begin{array}{ll}
L_2(s)&=sk+u-r\\
      &=(s-2)k+u-r+2k\\
      &>(s-2)k+u\\
      &\geq (s-2)\ello'+u\\
      &=L_3(s).
\end{array}\]
This contradicts the fact that for $s_1<s\leq s_2$, we have $\corank A^{s}=L_2(s)$, and therefore $L_2(s)\leq L_3(s)$.

Thus $|U_{bottom}|=u-r+2k\leq u$.

Finally, $P$ contains a middle $U$-chain only if $e_1$ and $e_2$ are both positive and $s_2=s_1+2$. In this case 
\[
\begin{array}{ll}
    |U_{middle}|&=s_2(k-e_2+e_1)-e_1+2e_2\\ \\ 
                &=s_2k-s_2(e_2-e_1)-e_1+2e_2\\ \\ 
                &=s_2k+e_2-(s_2-1)(e_2-e_1)\\ \\ 
                &=u-(s_2-1)(e_2-e_1).
\end{array}
\]
Since $s_2=s_1+2$, we have
\[
\begin{array}{ll}
         e_2-e_1&=(u-s_2k)-(u-r-s_1\ello)\\
         &=(s_2-2)\ello-s_2k+r\\
         &\geq (s_2-2)\ello'-s_2k+r\\
         &=L_3(s_2)-L_2(s_2).
\end{array}
\]
Since $\corank A^{s_2}=L_2(s_2)$, we have $L_3(s_2)\geq L_2(s_2)$. This in particular implies that $e_2-e_1\geq 0$ and therefor $|U_{middle}|\leq u$.  

This shows that in this case $P=([u]^k,[u-r]^\ello)$ with $Q(P)=(u,u-r)$. Thus, in this case $P=P_{k,\ello}$ in the table $\mathcal T(Q), \, Q=(u,u-r)$, as desired.

\bigskip

\noindent\underline{\bf Case B. ($L_1 \to L_3 \to L_4$)}
Suppose that $L_2(s)\geq L_!(s)\oplus L_2(s)\oplus L_4(s)$, for $s\geq 2$. Then  
\begin{equation}\label{bigkequ}
\begin{array}{ll}
\corank A^s&=s(k+\ello)\oplus (s\ello'+u-2\ello')\oplus(u+u-r)\\
           &=L_1(s)\oplus L_3(s)\oplus L_4(s).
\end{array}
\end{equation}

Let
\[
\begin{array}{ll}
s_3&=\lfloor x_{1,3}\rfloor=\lfloor\frac{u-2\ello'}{k+\ello-\ello'}\rfloor\\
e_3&=(u-2\ello')-s_3(k+\ello-\ello'), \mbox{ and}\\ \\

s_4&=\lfloor x_{3,4}\rfloor-2=\lfloor\frac{u-r}{\ello'}\rfloor\\
e_4&=u-r-s_4\ello'.

\end{array}\]

By equation~\ref{bigkequ} we have

\[
\corank A^s=
	\left\{
	\begin{array}{ll}
	s(k+\ello),& \mbox{ if }s=1, \dots, s_3\\
	(s-2)\ello'+u,&\mbox{ if }s= s_3+1, \dots, s_4+2\\
	u+u-r,&\mbox{ if }s\geq  s_4+3,
	\end{array}
	\right.
\]

We note that since $\corank A^s\leq L_2(s)$ always holds, we in particular have $u+u-r\leq L_2(x_{3,4})$. Thus $u\leq k\left(\frac{u-r}{\ello'}-2\right)$. Thus $\frac{u}{k}\leq \frac{u-r}{\ello'}-2$. On the other hand, since $\ello\geq \ello'\geq 1$, we also have $\frac{u-2\ello'}{k+\ello-\ello'}\leq \frac{u}{k}$. This in particular shows that $s_3\leq s_4$.  See figure~\ref{caseBfig} for a visualization of the linear terms affecting $\corank A^s$ in this case.

\begin{figure}

\boxed{\includegraphics[scale=.3]{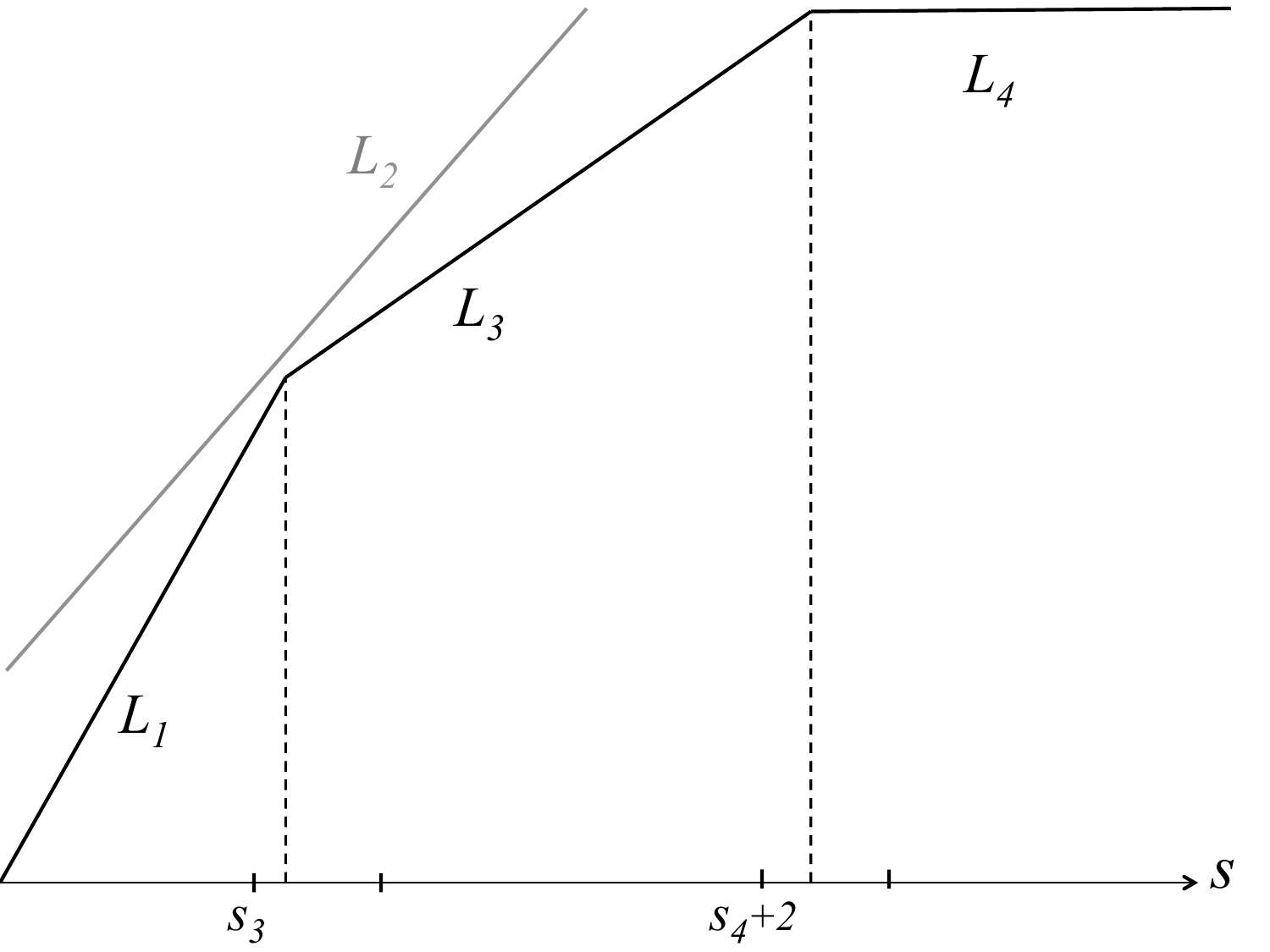}}

\caption{Comparing terms in the expression for $\corank A^s$ in Case B of the proof of Theorem~\ref{equTOpar}.}\label{caseBfig}
\end{figure}

Let $P$ be the Jordan type of $A$. Then

\[P^{\vee}=\Big(\, (k+\ello)^{s_3}, \, \ello'+e_3 , \,
(\ello')^{s_4-s_3+1},  \,  e_4\Big),\]
and

\[P=\left\{\begin{array}{ll}
        \Big(\, (s_4+3)^{e_4},\, (s_4+2)^{\ello'-e_4}, \, (s_3+1)^{e_3}, \, (s_3)^{k+\ello-\ello'-e_3}\, \Big),& \mbox{if }e_3,e_4>0\\\\

        \Big(\, (s_4+3)^{e_4},\, (s_4+2)^{\ello'-e_4}, \, (s_3)^{k+\ello-\ello'}\, \Big),& \mbox{if }e_4>0 \mbox{ and }e_3=0\\ \\

        \Big(\, (s_4+2)^{\ello'}, \, (s_3+1)^{e_3}, \, (s_3)^{k+\ello-\ello'-e_3}\, \Big),& \mbox{if }e_4=0 \mbox{ and }e_3>0\\ \\

        \Big(\, (s_4+2)^{\ello'},  \, (s_3)^{k+\ello-\ello'}\, \Big),& \mbox{if } e_3=e_4=0.
           \end{array}\right.\]
In all cases above, we can write $P=([u-r+2\ello']^{\ello'}, [u-2\ello']^{k+\ello-\ello'}).$ Looking at the $U$-chains in $P$, we have

\[\begin{array}{ll}
    |U_{top}| &=u-r+2\ello'  \\
    |U_{bottom}|&=u. 
\end{array}
\]

Since $\corank A^{s_3+1}=L_3(s_3+1)$, we have $L_3(s_3+1)\leq L_2(s_3+1)$. On the other hand, since $k+\ello'\leq r$, we have 
\[
\begin{array}{lll}
    L_3(s_3+1)&=(s_3+1)\ello'+u-2\ello'\\
              &=s_3\ello'+u-\ello'\\
              &\geq s_3\ello'+u-(r-k)\\
              &=s_3\ello'+k+u-r\\
              &=(s_3+1)k+u-r-s_3(k-\ello')\\
              &=L_2(s_3+1)-s_3(k-\ello').
\end{array}
\]
Thus $k\geq \ello'$ and consequently $r-2\ello'\geq 0$. This shows that $|U_{top}|\leq u$.

Moreover, partition $P$ contains a middle $U$-chain only if $e_3$ and $e_4$ are both positive and $s_4=s_3$. In this case $\corank A^{s_4+1}=L_3(s_4+1)$, and therefore, $L_3(s_4+1)\leq L_2(s_4+1)$. We also have 
\[
\begin{array}{ll}
    L_2(s_4+1)-L_3(s_4+1)&=((s_4+1)k+u-r)-((s_4+1)\ello'+u-2\ello')\\  
                         &=s_4(k-\ello')+k+\ello'-r\\
                         &=s_4(k+\ello-\ello')-s_4\ello+k\ello'-r\\
                         &\leq s_4(k+\ello-\ello')-s_4\ello'+k+\ello-r\\
                         &=(u-2\ello'-e_3)-(u-r-e_4)+k+\ello-r\\
                         &=e_4-e_3+k+\ello-2\ello'
\end{array}
\]
Thus $e_3-e_4\leq k+\ello-2\ello'$, and therefore
\[
\begin{array}{ll}
    |U_{middle}|&=(s_3+1)(\ello'-e_4+e_3)+\ello'-e_4+2e_4\\ 
                &=s_3(\ello'+e_3-e_4)+2\ello'+e_3\\
                &\leq s_3(\ello'+k+\ello-2\ello')+2\ello'+e_3\\
                &=s_3(k+\ello-\ello')+e_3+2\ello'=u.
\end{array}
\]
This shows that in this case $P=([u-r+2\ello']^{\ello'},[u-2\ello']^{k+\ello-\ello'})$ with $Q(P)=(u, u-r)$. Thus in this case $P=P_{k, \ello}$ in the table of $Q=(u,u-r)$, as desired.
\bigskip

\noindent\underline{\bf Case C.}
Suppose that neither Case A nor Case B holds. Then the following two sets are nonempty. 
\[\begin{array}{ll}
\mathcal{L}_2&=\{s\in \mathbb{Z}\, |\, s\geq 2\mbox{ and }L_2(s)<L_1(s)\oplus L_3(s)\oplus L_4(s)\}\\
\mathcal{L}_3&=\{s\in \mathbb{Z}\, |\, s\geq 2\mbox{ and }L_3(s)<L_1(s)\oplus L_2(s)\oplus L_4(s)\}
\end{array}\]

For the two sets above to be both nonempty, we need $L_2$ and $L_3$ must intersect at a point with $s$-coordinate at least two. Thus in this case, $k\neq \ello'$, other wise the lines $L_2$ and $L_3$ are parallel, or identical in the case of $k=\ello'=\frac{r}{2}$. Moreover, the $s$-coordinate of $L_2\cap L_3$ is $x_{2,3}=\frac{r-2\ello'}{k-\ello'}$. Thus we have $\frac{r-2\ello'}{k-\ello'}\geq 2$. If $k<\ello'$ then we get $r\leq 2k$. However this is impossible because $k+\ello'\leq r$. Thus we have $k>\ello'$. Since $\frac{r-2\ello'}{k-\ello'}\geq 2$, we have $r\geq 2k$, which in particular implies that $\ello'<r-k$. Thus since by definition $\ello'=\min\{r-k, \ello\}$, we must have $\ello=\ello'<k\leq r-k$.

Since $\ello<k\leq r-k$, we have $u-2\ello>u-r$. Thus $L_3$ is originally above $L_2$ until they intersect at $x_{2,3}$, after which $L_3$ will be below $L_2$. Therefore, 
\[
\corank A^s=
	\left\{
	\begin{array}{ll}
	s(k+\ello),& \mbox{ if }s=1, \dots, \lfloor x_{1,2}\rfloor\\
	sk+u-r,&\mbox{ if }s= \lfloor x_{1,2}\rfloor+1, \dots, \lfloor x_{2,3}\rfloor\\
	s\ello+u-2\ello,&\mbox{ if }s= \lfloor x_{2,3}\rfloor+1, \dots, \lfloor x_{3,4}\rfloor\\
	u+u-r,&\mbox{ if }s\geq  \lfloor x_{3,4}\rfloor+1,
	\end{array}
	\right.
\]

We also have $ x_{3,4}=\frac{u-r}{\ello}+2=x_{1,2}$. Thus if we let $s_1=\lfloor\frac{u-r}{\ello}\rfloor$, then we have $\lfloor x_{2,3}\rfloor=s_1+1$, and 
\[  \begin{array}{ll}
    \corank A^s      &=s(k+\ello),\mbox{ for }s=1, \dots, s_1\\
	\corank A^{s_1+1}&=(s_1+1)k+u-r,\\
	\corank A^{s_1+2}&=s_1\ello+u,\\
	\corank A^{s}    &=u+u-r, \mbox{ for }s\geq  s_1+2.
	\end{array}
\]
See figure~\ref{caseCfig} for a visualization of the linear terms affecting $\corank A^s$ in this case.

\begin{figure}

\boxed{\includegraphics[scale=.3]{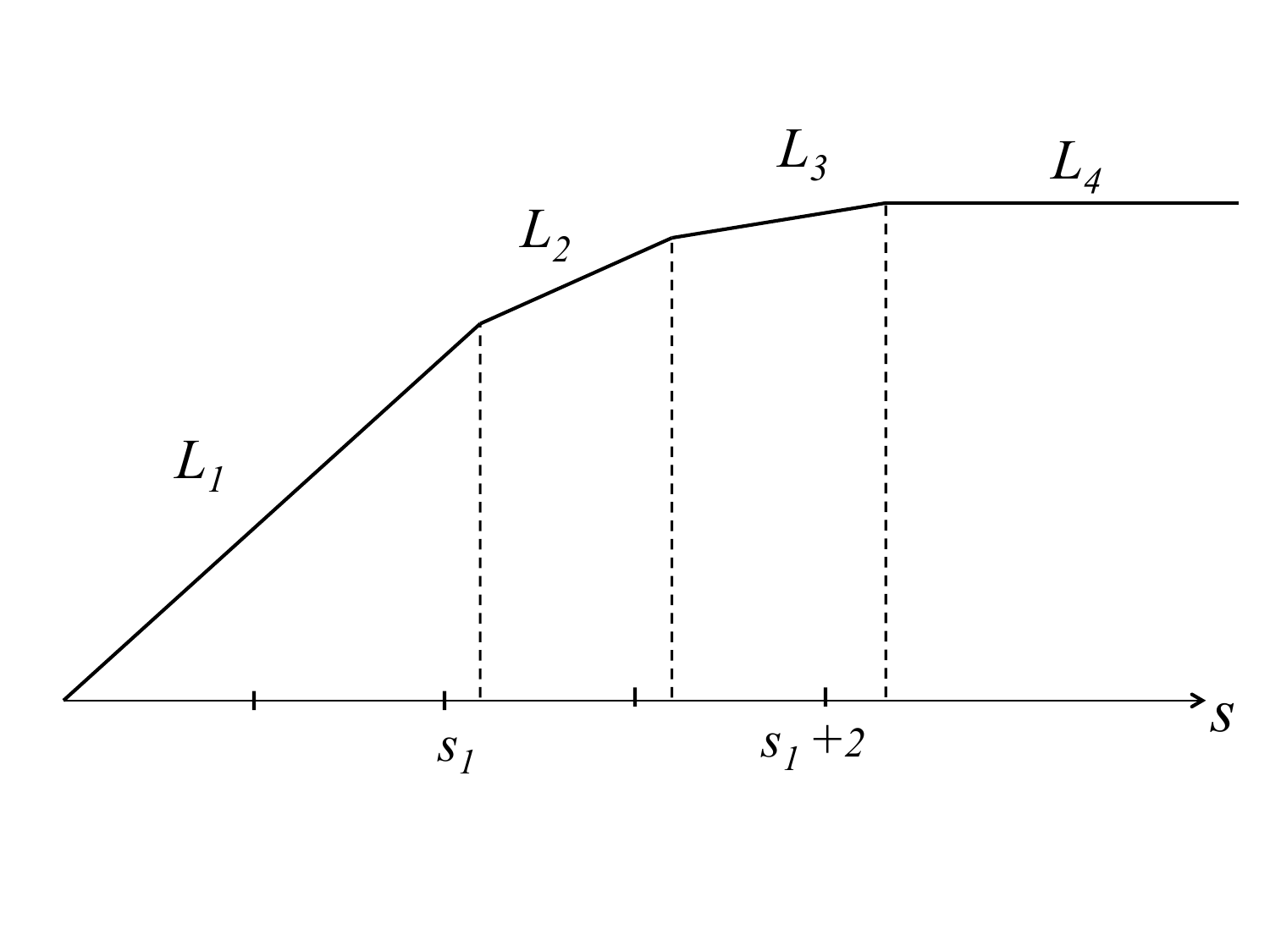}}

\caption{Comparing terms in the expression for $\corank A^s$ in Case C of the proof of Theorem~\ref{equTOpar}.}\label{caseCfig}
\end{figure}

Thus, for the Jordan type $P$ of $A$ we have 
\[P^{\vee}=\Big(\, (k+\ello)^{s_1}, \, k+e 
,  \, \ello+f, \, e \Big),\] where 
\[\begin{array}{ll}
e&=(u-r)-s_1\ello\\
f&=(r-2\ello)-(s_1+1)(k-\ello).
\end{array}\]

We have
\[P=\left\{
\begin{array}{ll}
        \Big(\, (s_1+3)^{e},\, (s_1+2)^{\ello+f-e}, \, (s_1+1)^{k-\ello+e-f}, \, (s_1)^{\ello-e}\, \Big),& \mbox{if }e>0\\\\

        \Big(\, (s_1+2)^{\ello+f},\, (s_1+1)^{k-\ello-f}, \, (s_1)^{\ello}\, \Big),& \mbox{if }e=0.
\end{array}\right.
\]
   
If $e=0$ then $P=([u]^{k},[u-r]^{\ello})$ and 
\[\begin{array}{ll}
    |U_{top}| &= (s_1+1)k+\ello+f\\
                   &=(s_1+1)(k-\ello)+(s_1+2)\ello+f\\
                   &=r-2\ello+u-r+2\ello\\
                   &=u, \\ \\

    |U_{bottom}|&=s_1(k-f)+k-\ello-f+2(\ello+f)\\
                &=(s_1+1)k-s_1f+k+\ello+f\\
                &=u-s_1f\\
                &\leq u.
\end{array}
\]
Thus, in this case $Q(P)=(u,u-r)$ and $P$, which is a type A partition, is the same as $P_{k,\ello}$ in the table $\mathcal T(Q), Q=(u,u-r)$, as desired.   
           
Finally, assume that $e>0$. Then we have
\[\begin{array}{ll}
    |U_{top}| &=(s_1+2)(\ello+f)+e \\
              &=s_1\ello +2\ello +(s_1+2)f+e\\
              &=u-r+2\ello+(s_1+1)f+f\\
              &< u-r+2\ello+(s_1+1)(k-\ello)+f\\
              &=u-r+2\ello+r-2\ello\\
              &=u,\\\\
              
    |U_{middle}| &= (s_1+1)k+\ello+f-e+2e\\
    		        &=(s_1+1)(k-\ell)+(s_1+2)\ello+f+e\\
		        &=r-2\ello+u-r+2\ello\\
		        &=u,\\\\

    |U_{bottom}|&=s_1(k-f)+k-\ello+e-f+2(\ello+f)\\
                &=(s_1+1)k-s_1f+k+\ello+e+f\\
                &=u-s_1f\\
                &\leq u.
\end{array}
\]

Thus, in this case, $P$ is a type C (also type B if $f=0$) partition with $Q(P)=(u,u-r)$. Moreover, we can write

\[\begin{array}{ll}
P&=\Big(\, (s_1+3)^{e},\, (s_1+2)^{\ello-e}, \, (s_1+2)^f, \, (s_1+1)^{k-\ello+e-f}, \, (s_1)^{\ello-e}\, \Big)\\
 &=\Big([u-r+2\ello]^\ello, \, [u-2\ello-s_1(\ello-e)]^{k-\ello+e}, \, (s_1)^{\ello-e}\, \Big).
 \end{array}\]
This shows that $P$ is in fact equal to the partition $P_{k, \ello}$ in the table $\mathcal T(Q),\, Q=(u,u-r)$, as desired.

\end{proof}
\begin{corollary}\label{identifycor} Let $Q=(u,u-r), r\ge 2.$  The partition $P_{k,\ello}(Q)$ in the table $\mathcal T(Q)$ of Theorem \ref{tablethm} is identical with the partition $P^Q_{k,\ello}$ of \cite[Equation 3]{BIK}, whose Burge code is $\alpha^{u-r-\ello}\beta^{\ello}\alpha^{r-k}\beta^k\alpha$ in the table $\mathcal D^{-1}(Q)$.
\end{corollary}\label{identcor}
\begin{proof}  The equations $E^Q_{k,\ello}$ defining each locus are identical.
\end{proof}
\begin{corollary}\label{equationcor} Let $Q=(u,u-r), r\ge 2$. The closure of the locus $W_P, P=P_{k,\ello}(Q)$ of elements in $\mathcal N(Q)$ having Jordan type $P$, is defined by the equations $E^Q_{k,\ello}$.
\end{corollary}
\begin{proof} This follows from Corollary \ref{identcor} and the main result Theorem 1.3 of \cite{BIK}.
\end{proof}
\vskip 0.2cm\par
These results, together with those of \cite{BIK}, identify the loci of $P_{k,\ello}(Q)$ in the table $\mathcal T(Q)$ with the locus defined by $E^Q_{k,\ello}$ - as stated in 
\cite[Conjecture 4.19]{IKVZ1}, see also \cite[Remark 4.13]{IKVZ2}.\vskip 0.2cm\par
 \thanks {\bf Acknowledgment}: 
We thank Bart Van Steirteghem, and Rui Zhao who participated in our early discussions of equations loci, beginning in \cite[\S 4.3]{IKVZ1}; see also \cite[Remark 4.13]{IKVZ2}.

\vskip 0.2cm
\begin{thebibliography}{ABDEF}
\bibitem[Ba1]{Ba1}
R. Basili, \emph{On the irreducibility of commuting varieties of nilpotent matrices}, J. Algebra 268 (2003), no. 1, 58--80.
\bibitem[Ba2]{Ba2}
R. Basili, \emph{On the maximal nilpotent orbit that intersects the centralizer of a matrix}, Transform. Groups 27 (2022), no. 1, 1--30. MR 4400714
\bibitem[BIK]{BIK}
M. Boij, A. Iarrobino, and L. Khatami, \emph{Jordan Type stratification of spaces of commuting nilpotent matrices}, preprint (2024).
\bibitem[Bur]{Bur}
W. H. Burge: \emph{A correspondence between partitions related to generalizations of the Rogers-Ramanujan identities}, Discrete Math. 34 (1981), 9-15.
\bibitem[IKVZ1]{IKVZ1}
A. Iarrobino, L. Khatami, B.Van Steirteghem, R. Zhao: \emph{Nilpotent matrices having a given Jordan type as maximum commuting nilpotent orbit}, arXiv:math.RA/1409.2192 v.2 (2015)
\bibitem[IKVZ2]{IKVZ2}
A. Iarrobino, L. Khatami, B.Van Steirteghem, R. Zhao: \emph{Nilpotent matrices having a given Jordan type as maximum commuting nilpotent orbit}, Lin. Alg. Appl. 546 (2018), 210--260.

\bibitem[IKM]{IKM}
J. Irving, T. Ko{\v{s}}ir, and M. Mastnak: \emph{A Proof of the Box Conjecture for Commuting Pairs of Matrices}, arXiv:math.CO/2403.18574
\bibitem[Kh1]{Kh1}
L. Khatami: \emph{The smallest part of the generic partition of the nilpotent commutator of a nilpotent
matrix}, J. Pure Appl. Algebra 218 (2014), no. 8, 1496--1516.

\bibitem[Kh2]{Kh2}
L. Khatami: \emph{Commuting Jordan types: a survey}, to appear in ``Deformation of Artinian algebras and Jordan type'', A.Iarrobino, P. Macias Marques, M. E. Rossi, and J. Vall\'{e}s, editors, Contemp. Mathematics \#805, p. 27-40, (2024).  arXiv:math.AC/2304.08550 (2023).
\bibitem[KO]{KO}
T. Ko{\v{s}}ir and P. Oblak: \emph{On pairs of commuting nilpotent matrices}, Transform. Groups 14 (2009), no. 1, 175--182.
\bibitem[Ob]{Ob}
P. Oblak \emph{On the nilpotent commutator of a nilpotent matrix},
Multilinear Algebra 60 (2012), no. 5, 599--612.

\bibitem[Pan]{Pan}
D. I. Panyushev: \emph{Two results on centralisers of nilpotent elements},
J. Pure and Applied Algebra, 212 no. 4 (2008), 774--779.
\end{thebibliography}
\end{document}